\documentclass[11pt,a4paper,leqno]{article}
\usepackage{amsmath}  
\usepackage{amsthm}		
\usepackage{bbm}      
\usepackage{amssymb}  
\usepackage{ifthen}   
\usepackage[T1]{fontenc}  
\usepackage[latin1]{inputenc}  

\advance\topmargin by -2.2cm
\newcommand{\cst}{\texttt{c{\!}s{\!}t}\,}
\renewcommand{\d}{\partial}

\newcommand{\eps}{\varepsilon}
\let\oddtheta\theta
\let\theta\vartheta
\let\vartheta\oddtheta
\let\oddphi\phi
\let\phi\varphi
\let\varphi\oddphi
\let\oddrho\rho
\let\rho\varrho
\let\varrho\oddrho

\newcommand{\N}{\mathbb N}

\newcommand{\R}{\mathbb R}

\renewcommand{\L}{\mathbb L}
\newcommand{\E}{\mathbb E} 
\renewcommand{\P}{\mathbb P} 

\newcommand{\cB}{\mathcal B}

\newcommand{\cD}{\mathcal D}

\newcommand{\cF}{\mathcal F}

\newcommand{\cL}{\mathcal L}

\newcommand{\cN}{\mathcal N}

\DeclareMathOperator{\diag}{diag}

\newcommand{\abs}[1]{\left|#1\right|}
\newcommand{\norm}[1]{\left\|#1\right\|}

\newcommand{\inpro}[2]{\left\langle#1,#2\right\rangle}

\newtheorem{thm}{Theorem}[section]

\newtheorem{lem}[thm]{Lemma}

\theoremstyle{definition}

\newtheorem{example}[thm]{Example}
\newtheorem{remark}[thm]{Remark}

\title{Local Asymptotic Normality for Shape and Periodicity in the Drift of a Time Inhomogeneous Diffusion\footnote{The final publication is available at Springer via http://dx.doi.org/10.1007/s11203-017-9157-5.}}
\author{Simon Holbach\footnote{Institut f{\"u}r Mathematik, Johannes Gutenberg-Universit{\"a}t Mainz, Staudingerweg 9, 55099 Mainz, Germany, e-mail: s.holbach$@$uni-mainz.de}}

\begin{document}
\maketitle
\begin{abstract}
We consider a one-dimensional diffusion whose drift contains a deterministic periodic signal with unknown periodicity $T$ and carrying some unknown $d$-dimensional shape parameter $\theta$. We prove Local Asymptotic Normality (LAN) jointly in $\theta$ and $T$ for the statistical experiment arising from continuous observation of this diffusion. The local scale turns out to be $n^{-1/2}$ for the shape parameter and $n^{-3/2}$ for the periodicity which generalizes known results about LAN when either $\theta$ or $T$ is assumed to be known.
\end{abstract}

\small{\textbf{Keywords:} local asymptotic normality, parametric signal estimation, periodic diffusion

}

\section{Introduction}

The center of our study is a one-dimensional diffusion $\xi$ following the stochastic differential equation
\begin{equation}
	d\xi_t = [S_{(\theta,T)}(t)+b(\xi_t)]dt + \sigma(\xi_t)dW_t, \quad t\in[0,\infty),
\label{diffusion}
\end{equation}
where $W$ is a one-dimensional Standard Brownian Motion, $b,\sigma\colon\R\to\R$ are measurable drift and volatility functions and $S_{(\theta,T)}\colon\R\to\R$ is a continuous signal that is parametrized by its periodicity $T$ and a $d$-dimensional shape parameter $\theta$. Taking $b\equiv0$, $\sigma\equiv1$ leads to the classical 'signal in white noise' model, which arises in a wide variety of fields including communication, radiolocation, seismic signal processing or computer-aided diagnosis and has been the subject of extensive study. For this special case, Ibragimov and Khasminskii (\cite{Ibra}) proved LAN with rate $n^{-3/2}$ for a smooth signal with known $\theta$ and discussed asymptotic efficiency for certain estimators. Golubev extended their approach with $\L^2$-methods in order to estimate $T$ at the same rate for unknown shape, which in turn was the basis for Castillo, L\'{e}vy-Leduc and Matias (\cite{CLM}) for nonparametric estimation of the shape under unknown $T$. For our more general diffusion \eqref{diffusion}, we will stay within the confines of parametric estimation. Our main assumptions are some $\L^2$-smoothness of the signal with respect to the parameters and positive Harris recurrence of the grid chain $(\xi_{nT})_{n\in\N}$. We prove LAN for the sequence of statistical experiments corresponding to continuous observation of $\xi$ over large time intervals with unknown $\theta$ and $T$. H{\"o}pfner and Kutoyants have solved this problem both for known $T$ with unknown $\theta$ (\cite{HK1}) and for known $\theta$ with unknown $T$ (\cite{HK3}). Our result extends both of these and allows for application to simultaneous estimation of shape and periodicity, as under LAN we can use H\'{a}jek's Convolution Theorem and the Local Asymptotic Minimax Theorem in order to establish optimality for estimators, when the rescaled estimation errors are stochastically asymptotically equivalent to the central statistic of the experiment (see \cite{LeCam}, \cite {Davies}, \cite{Kut} or \cite{HoBo} for a detailed presentation of the relevant theory).

\section{Precise Assumptions and Results}

Now we will give and explain the exact setting in which we would like to work in this paper. Let $\Theta\subset\R^d$ be an open set. First, consider the following basic hypotheses:
\begin{itemize}
	\item[(H1)] For each $(\theta,T)\in\Theta\times(0,\infty)$, the equation \eqref{diffusion} has a unique strong solution.
	\item[(H2)] $\sigma$ is bounded away from zero.
\end{itemize}
We write $\P^{(\theta,T)}$ for the law on $C([0,\infty))$ under which the canonical process $(\eta_t)_{t\ge 0}$ is the solution of \eqref{diffusion} issued from some fixed and deterministic starting point $\xi_0 \in \R$ with the parameters $(\theta,T)\in\Theta\times(0,\infty)$. We define
\[
	\cF_t:=\sigma(\eta_s \,|\, 0\le s\le t+):=\bigcap_{r>t} \sigma(\eta_s \,|\, 0\le s\le r),
\]
the $\sigma$-algebra generated by observation of $\eta$ up to time $t+$, $t\ge0$. Note that the drift coefficient of \eqref{diffusion} depends on time and on the parameter $(\theta,T)\in\Theta\times(0,\infty)$, while the diffusion coefficient depends on neither. Therefore we can use \cite[Theorem 6.10]{HoBo} to calculate the log-likelihood-ratio
\[
	 \Lambda_t^{(\tilde\theta,\tilde T)/(\theta,T)}:= \log\left(\frac{d\P^{(\tilde\theta,\tilde T)}|_{\cF_t}}{d\P^{(\theta,T)}|_{\cF_t}}\right)= \int_0^t \frac{S_{(\tilde\theta,\tilde T)}(s)-S_{(\theta,T)}(s)}{\sigma(\eta_s)}dW_s - \frac{1}{2}\int_0^t \left(\frac{S_{(\tilde\theta,\tilde T)}(s)-S_{(\theta,T)}(s)}{\sigma(\eta_s)}\right)^2ds.
\]
Our goal is to prove local asymptotic normality for the sequence of experiments given by
\[
	\left(C([0,\infty)), \cF_n, \left\{ \P^{(\theta,T)}|_{\cF_n} \, \middle| \, (\theta,T) \in \Theta\times(0,\infty) \right\}\right), \quad n \in \N,
\]
and to that end we will now give more precise smoothness assumptions on the deterministic signal. 
\begin{enumerate}
	\item[(S1)] For each $\theta \in \Theta$, we have a 1-periodic function $S_\theta \in C^2([0,\infty))$.
	\item[(S2)] $S_\cdot(s) \in C^1(\Theta)$ for each $s \in [0,\infty)$.
	\item[(S3)] $\nabla_\theta S_\theta(\cdot) \in (\L^2_{\text{loc}}(0,\infty))^d$ for each $\theta\in\Theta$.
	\item[(S4)] The mapping
				\[
					S \colon \Theta \times (0,\infty) \to \L^2_{\text{loc}}(0,\infty), \quad (\theta,T) \mapsto S_{(\theta,T)}:=S_\theta\left(\frac{\cdot}{T}\right)
				\]
				is $\L^2_{\text{loc}}$-differentiable with derivative
				\[
					\dot{S} \colon \Theta \times (0,\infty) \to (\L^2_{\text{loc}}(0,\infty))^{d+1}, \quad (\theta,T) \mapsto \dot{S}_{(\theta,T)}:=\left(\begin{array}{c}\d_{\theta_1}S_{(\theta,T)} \\ \vdots \\ \d_{\theta_d}S_{(\theta,T)} \\ \d_T S_{(\theta,T)} \end{array}\right)
				\]	
				in the sense that for every $t>0$ and $(\theta,T) \in \Theta \times (0,\infty)$ we have
				\[
					\int_0^t \left(\frac{S_{(\theta,T)}(s)-S_{(\tilde\theta,\tilde T)}(s)-((\theta,T)-(\tilde\theta,\tilde T))^\top \dot{S}_{(\theta,T)}(s)}{\abs{(\theta,T)-(\tilde\theta,\tilde T)}} \right)^2ds \to 0, \text{ as } (\tilde\theta,\tilde T) \to (\theta,T).
				\]
	\item[(S5)] $\dot{S}$ is $\L^2_{\text{loc}}$-continuous in the sense that for all $t>0$ and $(\theta,T) \in \Theta \times (0,\infty)$ we have
				\[
					\int_0^t \abs{\dot{S}_{(\theta,T)}(s)-\dot{S}_{(\tilde\theta,\tilde T)}(s)}^2ds \to 0, \text{ as } (\tilde\theta,\tilde T) \to (\theta,T).
				\]
	\item[(S6)] The mapping $(0,\infty) \ni T \mapsto \nabla_\theta S_{(\theta,T)} \in (\L^2_{\text{loc}}(0,\infty))^d$ satisfies the following local H{\"o}lder condition: For each $\theta \in \Theta$ and $T>0$ there are $\alpha \in (0,2]$ and $\beta\in[0,1+3\alpha/2)$ such that for suitable $\eps>0$ and $t_0\ge0$
\[
	\int_{t_0}^t \abs{\nabla_\theta S_{(\theta,T)}(s) - \nabla_\theta S_{(\theta,\tilde T)}(s)}^2 ds \le Ct^\beta\abs{T-\tilde T}^\alpha
\]
for all $t>t_0$, $\tilde T \in (T-\eps,T+\eps)$ and some constant $C$ that does not depend on $\tilde T$ or $t$.
\end{enumerate}

\begin{remark} 1.) We observe that if (S1) - (S3) hold and $\dot S_{(\theta,T)}(s)$ is continuous (and thus also locally bounded) in $\theta$, $T$ and $s$, (S4) and (S5) are immediate by dominated convergence. Note that in general, (S1) - (S3) do not require that for example $\d_{\theta_1}S_{(\theta,T)}(s)$ is continuous (or even locally bounded) in $T$ or $s$.

2.) If for every $\theta \in \Theta$, $T>0$ and $t>0$ there are $\delta=\delta(\theta,T) \in (0,1]$ and $C(\theta,t) \le \cst t^\zeta$ with $\zeta \in [0,\delta/2)$ such that the mapping $[0,\infty) \ni s \mapsto \nabla_\theta S_\theta(s)$ is H{\"o}lder-$\delta$-continuous on $[0,t]$ with H{\"o}lder-constant $C(\theta,t)$, we get that for sufficiently small $\eps>0$ and for all $\tilde T \in (T-\eps,T+\eps)$
\begin{align*}
	\int_0^{t} \abs{\nabla_\theta S_{(\theta,\tilde T)}(s)-\nabla_\theta S_{(\theta,T)}(s)}^2ds & = \int_0^{t} \abs{\nabla_\theta S_\theta\left(\frac{s}{\tilde T}\right)-\nabla_\theta S_\theta\left(\frac{s}{T}\right)}^2ds \\
	&\le \sup_{T'\in(T-\eps,T+\eps)}C\left(\theta,\frac{t}{T'}\right) \int_0^{t} \abs{\frac{s}{\tilde T}-\frac{s}{T}}^{2\delta(\theta,T)}ds \\
	& \le \cst \left(\frac{t}{T-\eps}\right)^{2\zeta} \left(\frac{\abs{\tilde T-T}}{(T-\eps)^2}\right)^{2\delta} \int_0^ts^{2\delta}ds,	
\end{align*}
which implies the H{\"o}lder condition (S6) with $\alpha=2\delta$ and $\beta=2(\delta+\zeta)+1 \in [0,1+3\alpha/2)$.

3.) As a consequence of the two preceding remarks, all of the hypotheses (S1) - (S6) are fulfilled if the mapping $\Theta\times[0,\infty) \ni (\theta,s) \mapsto S_\theta(s)$ is in $C^2_b(\Theta\times[0,\infty))$. Existence and boundedness of $\d_s\nabla_\theta S_\theta(s)$ ensure that we can choose $\delta=1$ and $\zeta=0$ above.

4.) Let $S_\theta(s)=f(\theta,\phi(s))$ with $\phi \in C^2([0,\infty))$ $1$-periodic and $f \in C^{1,2}(\Theta\times\R)$. In particular, we have (S1) - (S3). Write $\nabla f=(f_1,\ldots,f_{d+1})$. Since
\[
	\dot S_{(\theta,T)}(s)=\left(\begin{array}{c}f_1\left(\theta,\phi\left(\frac{s}{T}\right)\right) \\ \vdots \\ f_d\left(\theta,\phi\left(\frac{s}{T}\right)\right) \\ f_{d+1}\left(\theta,\phi\left(\frac{s}{T}\right) \right) \phi'\left(\frac{s}{T}\right) \left(\frac{-s}{T^2}\right) \end{array}\right)
\]
is obviously continuous in $\theta$, $T$ and $s$, so we also have (S4) and (S5). Moreover we see that the H{\"o}lder property in 2.) is fulfilled if and only if it is fulfilled by the mapping $[0,\infty) \ni s \mapsto (f_1,\ldots,f_d)(\theta,s)$. So in that case all of the hypotheses (S1) - (S6) hold.

5.) A special case of the preceding example is a product structure $S_\theta(s)=g(\theta)\phi(s)$ with $\phi \in C^2([0,\infty))$ $1$-periodic and $g \in C^1(\Theta)$. As for all $s, \tilde s \in [0,\infty)$ we have
\[
	\abs{\nabla_\theta S_\theta(s)-\nabla_\theta S_\theta(\tilde s)} \le \abs{\nabla g(\theta)} \norm{\phi'}_\infty \abs{s-\tilde s},
\]
no further conditions are needed to ensure the H{\"o}lder property in 2.) to hold with $\delta=1$ and $\zeta=0$.

6.) Choosing $\phi(s)=\sin(2k\pi s)$ or $\phi(s)=\cos(2k\pi s)$ with $k \in \N_0$ in the above example and observing that our hypotheses are stable under linear combinations, we see that all of them are fulfilled for signals of the form
\[
	S_{(\theta,T)}(s)=\sum_{k=1}^l \left( g_k(\theta)\sin\left(\frac{2k\pi s}{T}\right) + h_k(\theta)\cos\left(\frac{2k\pi s}{T}\right) \right)
\]
with $l \in \N_0$ and $g_k,h_k \in C^1(\Theta)$ for all $k \in\{1,\ldots,l\}$.


\end{remark}

Let us now fix $(\theta,T)\in\Theta\times(0,\infty)$. As a consequence of the periodic structure in the drift term of the diffusion \eqref{diffusion}, its transition semi-group $\left(P^{(\theta,T)}_{s,t}\right)_{0\le s< t}$ under $\P^{(\theta,T)}$ has the property
\[
	P^{(\theta,T)}_{s+kT,t+kT}=P^{(\theta,T)}_{s,t} \quad \text{for all $t>s\ge 0$ and $k\in\N$.}
\]
Thus the grid chain $\left(\xi_{kT}\right)_{k\in\N_0}$ is a time homogeneous Markov chain with one-step transition kernel $P^{(\theta,T)}_{0,T}$. We suppose:
\begin{itemize}
	\item[(H3)] The grid chain under $\P^{(\theta,T)}$ is positive recurrent in the sense of Harris with invariant probability measure $\mu^{(\theta,T)}$.
\end{itemize}
Verifiable criteria for this condition can be found e.g. in \cite{HL1}, a specific example will be given at the end of this article.

The ergodicity assumption (H3) allows us to make use of certain variants of classical Limit Theorems (see \cite{HK2}, \cite{HK3}), which we will need for Lemma \ref{H(s)} below. With \cite[Lemma 2.1]{HK2} in mind, we define the measure
\begin{equation}
	\nu^{(\theta,T)}(ds)=\mu^{(\theta,T)} P^{(\theta,T)}_{0,sT}(\sigma^{-2}) ds \quad \text{ on $\cB((0,1))$,}
\label{measure}
\end{equation}
which is finite, as $\mu^{(\theta,T)}$ is finite and $\sigma$ is bounded away from zero by (H2). We write $\inpro{\cdot}{\cdot}_{\nu^{(\theta,T)}}$ for the standard inner product in $\L^2(\nu^{(\theta,T)})$. For each $t\ge0$ define the symmetric $(d+1)\times(d+1)$-dimensional block matrix
\begin{equation}
F_{(\theta,T)}(t):= \begin{pmatrix}
	t \left(\inpro{\d_{\theta_i}S_\theta}{\d_{\theta_j}S_\theta}_{\nu^{(\theta,T)}}\right)_{i,j=1,\ldots,d} & -\frac{t^2}{2T^2}\left(\inpro{\d_{\theta_i}S_\theta}{S_\theta'}_{\nu^{(\theta,T)}}\right)_{i=1,\ldots,d} \\
	 \cdots & \frac{t^3}{3T^4} \inpro{S_\theta'}{S_\theta'}_{\nu^{(\theta,T)}}
\end{pmatrix}.
\label{Matrix}
\end{equation}
Its derivative with respect to $t$,
\begin{align*}
F_{(\theta,T)}'(t)&=\begin{pmatrix}
	\left(\inpro{\d_{\theta_i}S_\theta}{\d_{\theta_j}S_\theta}_{\nu^{(\theta,T)}}\right)_{i,j=1,\ldots,d} & -tT^{-2}\left(\inpro{\d_{\theta_i}S_\theta}{S_\theta'}_{\nu^{(\theta,T)}}\right)_{i=1,\ldots,d} \\
	 \cdots & t^2 T^{-4}\inpro{S_\theta'}{S_\theta'}_{\nu^{(\theta,T)}}
\end{pmatrix} \\
&=\nu^{(\theta,T)}\left[\left(\begin{array}{c} \nabla_\theta S_\theta \\ -tT^{-2} S_\theta' \end{array}\right) \left(\begin{array}{c} \nabla_\theta S_\theta \\ -tT^{-2} S_\theta' \end{array}\right)^{\!\!\!\!\!\top} \right]
\end{align*}
is obviously symmetric and nonnegative-definite. We suppose that
\begin{itemize}
	\item[(S7)] $F_{(\theta,T)}'(t)$ is invertible for all $t>0$.
\end{itemize}
Note that (S7) is equivalent to linear independence of $\d_{\theta_1}S_\theta,\ldots,\d_{\theta_d}S_\theta,S'_\theta$ in $\L^2(\nu^{(\theta,T)})$.

\begin{example}\label{signal}
1.) Let $\sigma\equiv 1$, then $\nu^{(\theta,T)}$ is just Lebesgue's measure. Considering once again signals of the form
\[
	S_{(\theta,T)}(s)=\sum_{k=1}^l \left( g_k(\theta)\sin\left(\frac{2k\pi s}{T}\right) + h_k(\theta)\cos\left(\frac{2k\pi s}{T}\right) \right)
\]
with $l \in \N_0$ and $g_k,h_k \in C^1(\Theta)$ for all $k \in\{1,\ldots,l\}$, an elementary calculation yields
\[
	\left(F'_{(\theta,T)}\right)_{i,j}(t)=\begin{cases} \frac{1}{2} \sum_{k=1}^l\left(g_k^{(i)}(\theta)g_k^{(j)}(\theta)+ h_k^{(i)}(\theta)h_k^{(j)}(\theta)\right), & (i,j) \in \{1,\ldots,d\}^2, \\
																				-\pi tT^{-2} \sum_{k=1}^lk\left(g_k(\theta)h_k^{(j)}(\theta)- g_k^{(j)}(\theta)h_k(\theta)\right), & i=d+1, \, j \in \{1,\ldots,d\}, \\
																				2\pi^2t^2 T^{-4} \sum_{k=1}^lk^2\left(g_k(\theta)^2+ h_k(\theta)^2\right), & i=j=d+1, \end{cases}
\]
where a superscript $(i)$ indicates partial derivation with respect to $\theta_i$. Note that in the case that either the coefficients of the $\sin$-terms or those of the $\cos$-terms vanish identically, this matrix is invertible if and only if $\left(F'_{(\theta,T)}\right)_{d+1,d+1}(t)>0$ and $\d_{\theta_1}S_\theta,\ldots,\d_{\theta_d}S_\theta$ are linearly independent in $\L^2(\nu)$. In particular this is ensured if $l=d$ and if for all $k,i \in \{1,\ldots,d\}$ we have $h_k\equiv0$ and
\[
	g_k^{(i)}(\theta)\neq0 \text{ for $i=k$} \quad \text{and} \quad g_k^{(i)}(\theta)=0 \text{ else.}
\]
A simple example would be $g_k(\theta)=\theta_k$.

2.) Similarly, if the signal is of the form
\[
	S_{(\theta,T)}(s)=\sum_{k=1}^d \theta_k \phi_k\left(\frac{s}{T}\right),
\]
where $\phi_1,\ldots,\phi_d$ are orthonormal in $\L^2(\nu^{(\theta,T)})$, we have
\[
F_{(\theta,T)}'(t)=\begin{pmatrix}
	\mathbbm{1}_{d\times d} & -tT^{-2}\left(\sum_{j=1}^d\theta_j \inpro{\phi_i}{\phi_j'}_{\nu^{(\theta,T)}}\right)_{i=1,\ldots,d} \\
	 \cdots & t^2 T^{-4}\sum_{i,j=1}^d\theta_i\theta_j \inpro{\phi_i'}{\phi_j'}_{\nu^{(\theta,T)}}
\end{pmatrix},
\]
which is invertible whenever
\[
	\sum_{i,j=1}^d\theta_i\theta_j \inpro{\phi_i'}{\phi_j'}_{\nu^{(\theta,T)}} \neq \sum_{i,j=1}^d\left(\theta_j \inpro{\phi_i}{\phi_j'}_{\nu^{(\theta,T)}}\right)^2.
\]
\end{example}

\begin{thm}[Local asymptotic normality]\label{LAN}
Fix $(\theta,T)\in\Theta\times(0,\infty)$ and grant all of the hypotheses (H1) - (H3) and (S1) - (S7). Fix any bounded sequence $(h_n)_{n \in \N} \subset \R^{d+1}$ and set $(\theta_n,T_n):=(\theta,T)+\delta_n h_n$ with the local scale
\[
	\delta_n := \diag\left(n^{-1/2},\ldots,n^{-1/2},n^{-3/2}\right) \in \R^{(d+1)\times(d+1)} \quad \text{for all $n\in\N$.} 
\]
Then we have LAN
\[
	\Lambda_n^{(\theta_n,T_n)/(\theta,T)}=h_n^\top \Delta^{(\theta,T)}_n-\frac12 h_n^\top F^{(\theta,T)}h_n + o_{\P^{(\theta,T)}}(1), \quad n \to \infty,
\]
with Fisher Information $F^{(\theta,T)}=F_{(\theta,T)}(1)$ as introduced in \eqref{Matrix} and score
\[
	\Delta^{(\theta,T)}_n=\delta_n \int_0^n \frac{\dot{S}_{(\theta,T)}(s)}{\sigma\left(\eta_s\right)}dW_s \quad \text{for all $n\in\N$,} 
\]
such that weak convergence
\[
	\cL\left(\Delta^{(\theta,T)}_n\middle|\P^{(\theta,T)}\right) \xrightarrow{n\to\infty} \cN\left(0,F_{(\theta,T)}\right)
\]
holds.
\end{thm}

\begin{remark}
The above theorem naturally extends to the case of a $D$-dimensional signal, all of whose components satisfy (S1) - (S6), that is present in the drift of a $D$-dimensional diffusion driven by an $M$-dimensional Brownian Motion, $D,M\ge1$. Assumption (H2) has to be replaced by uniform ellipticity of $\sigma\sigma^\top$, where $\sigma\colon \R^D \to \R^{D\times M}$ is the volatility matrix of the corresponding diffusion equation. Lemma \ref{H(s)} below and assumption (S7) also need to be restated accordingly. Notation becomes far more complex in this case, but the general line of the proof remains unaltered.
\end{remark}

\textbf{Notational Convention:} For the remainder of this article, $(\theta,T)\in\Theta\times(0,\infty)$ will be \emph{fixed} and we drop corresponding indices (for $\mu, \nu, F, F', \ldots$) whenever there is no risk of ambiguity.

The proof of the Theorem makes use of the following Lemma, which is a simple consequence of Lemma 2.2 from \cite{HK3} and which we state explicitly for the sake of convenience. The case $k=0$ is not included in \cite{HK3}, but it follows easily with a simplified version of the same argument. The essential ingredient in this Lemma (and thus in Theorem \ref{LAN}) is a Strong Law of Large Numbers for certain functionals of path segments of Markov processes with the periodic ergodicity property (H3), see section 2 of \cite{HK2}.

\begin{lem}\label{H(s)}
Fix $(\theta,T)\in\Theta\times(0,\infty)$ and assume (H1) - (H3). For any 1-periodic bounded measurable function $f \colon [0,\infty) \to \R$ and $k \in \N_0$ we have
\[
	(k+1)t^{-(k+1)} \int_0^t s^k \frac{f\left(\frac{s}{T}\right)}{\sigma^{2}(\eta_s)}ds \xrightarrow{t \to \infty} \nu[f] \quad \text{$\P^{(\theta,T)}$-almost surely.}
\]
\end{lem}

\begin{proof}[Proof of Theorem \ref{LAN}]
1.) The main idea is to introduce a time step size $t>0$ into the log-likelihood and then for each $n \in \N$ interpret $\left(\Lambda_{tn}^{(\theta_n,T_n)/(\theta,T)}\right)_{t\ge0}$ as a continuous time stochastic process. Splitting it into several parts and applying the above Lemma together with tools from continuous time martingale theory will eventually lead to the desired quadratic expansion. Indeed, we write
\begin{align*}
	\Lambda_{tn}^{(\theta_n,T_n)/(\theta,T)}
	&= \int_0^{tn} \frac{S_{(\theta_n,T_n)}(s)-S_{(\theta,T)}(s)}{\sigma(\eta_s)}dW_s - \frac{1}{2}\int_0^{tn} \left(\frac{S_{(\theta_n,T_n)}(s)-S_{(\theta,T)}(s)}{\sigma(\eta_s)}\right)^2ds \\
	&= (\delta_nh_n)^\top \left( \int_0^{tn} \frac{\dot{S}_{(\theta,T)}(s)}{\sigma(\eta_s)}dW_s\right) \\
	& \quad - \frac12 (\delta_n h_n)^\top \left( \int_0^{tn} \frac{\dot{S}_{(\theta,T)}(s)\dot{S}_{(\theta,T)}(s)^\top}{\sigma^2(\eta_s)} ds \right) (\delta_nh_n) \\
	& \quad + \int_0^{tn} \frac{S_{(\theta_n,T_n)}(s)-S_{(\theta,T)}(s)-(\delta_nh_n)^\top\dot{S}_{(\theta,T)}(s)}{\sigma(\eta_s)}dW_s \\
	& \quad - \frac12\int_0^{tn} \left(\frac{S_{(\theta_n,T_n)}(s)-S_{(\theta,T)}(s)-(\delta_nh_n)^\top\dot{S}_{(\theta,T)}(s)}{\sigma(\eta_s)}\right)^2ds \\
	& \quad - \int_0^{tn} \frac{\big(S_{(\theta_n,T_n)}(s)-S_{(\theta,T)}(s)-(\delta_nh_n)^\top\dot{S}_{(\theta,T)}(s)\big) \big((\delta_nh_n)^\top\dot{S}_{(\theta,T)}(s)\big)}{\sigma^2(\eta_s)} ds \\
&=: h_n^\top \Delta_n(t)-\frac12 h_n^\top F_n(t)h_n+R_n(t)-\frac12 U_n(t)-V_n(t)
\end{align*}
and in order to prove the Theorem, we have to study convergence in distribution of $\Delta_n$ for $n \to \infty$ and show almost sure convergence of $F_n(1)$ to $F=F(1)$. Finally, we show that $R_n(t)$, $U_n(t)$ and $V_n(t)$ converge to zero in probability under $\P^{(\theta,T)}$.

2.) For any fixed $n\in\N$ the process
\[
	M_n:=(\Delta_n(t))_{t\ge0}=\left(\delta_n \int_0^{tn} \frac{\dot{S}_{(\theta,T)}(s)}{\sigma(\eta_s)}dW_s\right)_{t\ge0}
\]
is obviously an $\R^{d+1}$-valued local $\P^{(\theta,T)}$-martingale. In order to determine its weak limit for $n \to \infty$ in the Skorohod space $\cD([0,\infty);\R^{d+1})$, we first calculate its angle bracket. For $i,j \in \{1,\ldots,d\}$ we have
\[
	\inpro{M_n^i}{M_n^j}_t = \frac{1}{n} \int_0^{tn} \frac{\d_{\theta_i}S_{(\theta,T)}(s) \d_{\theta_j}S_{(\theta,T)}(s)}{\sigma^2(\eta_s)}ds
												 = t \cdot \frac{1}{tn} \int_0^{tn} \frac{\d_{\theta_i}S_\theta\left(\frac{s}{T}\right) \d_{\theta_j}S_\theta\left(\frac{s}{T}\right)}{\sigma^2(\eta_s)}ds
\]
and due to the periodicity of $S_\theta$ and by Lemma \ref{H(s)} with $k=0$ this expression converges to
\[
	t \nu[\d_{\theta_i}S_\theta\d_{\theta_j}S_\theta] = F_{i,j}(t)
\]
$\P^{(\theta,T)}$-almost surely for $n \to \infty$, where we used the notation introduced in \eqref{measure} and \eqref{Matrix}. The same argument with $k=1$ yields
\begin{align*}
	\inpro{M_n^j}{M_n^{d+1}}_t=\inpro{M_n^{d+1}}{M_n^j}_t &= n^{-2} \int_0^{tn} \frac{\d_TS_{(\theta,T)}(s)\d_{\theta_j}S_{(\theta,T)}(s)}{\sigma^2(\eta_s)}ds \\
														 &= \frac{-t^2}{2T^2} \cdot 2(tn)^{-2} \int_0^{tn} s \cdot \frac{S_\theta'\left(\frac{s}{T}\right) \d_{\theta_j}S_\theta\left(\frac{s}{T}\right)}{\sigma^2(\eta_s)}ds \\
														 & \xrightarrow{n \to \infty}  \frac{-t^2}{2T^2} \nu[S_\theta'\d_{\theta_j}S_\theta] = F_{{d+1},j}(t)=F_{j,{d+1}}(t)
\end{align*}
$\P^{(\theta,T)}$-almost surely and finally (with $k=2$)
\begin{align*}
	\inpro{M_n^{d+1}}{M_n^{d+1}}_t &= n^{-3} \int_0^{tn} \left(\frac{\d_TS_{(\theta,T)}(s)}{\sigma(\eta_s)}\right)^2ds = \frac{t^3}{3T^4} \cdot 3(tn)^{-3} \int_0^{tn} s^2 \cdot \frac{\left(S_\theta'\left(\frac{s}{T}\right)\right)^2}{\sigma^2(\eta_s)}ds \\
																 & \xrightarrow{n \to \infty} \frac{t^3}{3T^4} \nu[(S_\theta')^2] =F_{d+1,d+1}(t)
\end{align*}
$\P^{(\theta,T)}$-almost surely. In other words for all $t\ge0$ the quadratic variation $\langle M_n\rangle_t$ converges $\P^{(\theta,T)}$-almost surely to the matrix $F(t)$ as $n \to \infty$ and the Martingale Convergence Theorem \cite[Corollary VIII.3.24]{Jacod} implies weak convergence
\begin{equation}
	M_n \xrightarrow{\cL} M \quad \text{in $\cD([0,\infty);\R^{d+1})$}
\label{weak}
\end{equation}
to some limit martingale $M=(M(t))_{t\ge0}$. By (S7), $F'(t)$ is invertible and as it is symmetric and nonnegative-definite, it possesses a square root, i.e. there is some uniquely determined matrix $A=:\sqrt{F'(t)} \in \R^{(d+1)\times(d+1)}$ with $AA=F'(t)$. Thus the Representation Theorem \cite[Theorem II.7.1]{Ikeda} yields that $M$ can be expressed as
\[
	M(t)=\int_0^t \sqrt{F'(s)} dB_s, \quad t\ge0,
\]
with some $(d+1)$-dimensional standard Brownian motion $B$. Together with \eqref{weak} this also implies weak convergence
\[
	\cL\left(M_n(t)\middle|\P^{(\theta,T)}\right) \to \cL\left(M(t)\middle|\P^{(\theta,T)}\right)=\cN\left(0,\int_0^t F'(s)ds\right)=\cN\left(0,F(t)\right) \quad \text{for all $t\ge0$.}
\]
In particular we have weak convergence of $\Delta_n=M_n(1) \to \cN(0,F(1))=\cN(0,F)$.

3.) In the second step we have shown on the fly $\P^{(\theta, T)}$-almost sure convergence of $F_n(1) = \langle M_n\rangle_1$ to $\langle M\rangle_1=F(1)$ for $n\to\infty$.

4.) It remains to show convergence to $0$ of the remainder terms $R_n,U_n$ and $V_n$ introduced at the very beginning of this proof. Therefore, we consider the sequence
\[
	R_n:=(R_n(t))_{t\ge0}=\left(\int_0^{tn} \frac{S_{(\theta_n,T_n)}(s)-S_{(\theta,T)}(s)-(\delta_nh_n)^\top\dot{S}_{(\theta,T)}(s)}{\sigma(\eta_s)}dW_s\right)_{t\ge0}, \quad n\in\N,
\]
of local $\P^{(\theta,T)}$-martingales. Using that by (H2) the volatility $\sigma$ is bounded away from $0$, we estimate
\begin{align*}
 \langle R_n \rangle_t =U_n(t)&= \int_0^{tn} \left(\frac{S_{(\theta_n,T_n)}(s)-S_{(\theta,T)}(s)-(\delta_nh_n)^\top\dot{S}_{(\theta,T)}(s)}{\sigma(\eta_s)}\right)^2ds \\
& \le \frac{3}{\inf \sigma^2} \bigg( \int_0^{tn} \big(S_{(\theta_n,T_n)}(s)-S_{(\theta,T_n)}(s)-(\theta_n-\theta)^\top\nabla_\theta S_{(\theta,T_n)}(s)\big)^2ds \\
    & \quad + \int_0^{tn} \big((\theta_n-\theta)^\top(\nabla_\theta S_{(\theta,T_n)}-\nabla_\theta S_{(\theta,T)}(s))\big)^2ds \\
		& \quad + \int_0^{tn} \big(S_{(\theta,T_n)}(s)-S_{(\theta,T)}(s)-(T_n-T) \d_T S_{(\theta,T)}(s)\big)^2ds \bigg) \\
 & =:	\frac{3}{\inf \sigma^2} (A_n+B_n+C_n).
\end{align*}
Let $H:=\sup_{n\in\N}\abs{h_n}$. For sufficiently large $n\in\N$ we have $T_n\in[T/2,2T]$ and thus
\begin{align*}
	A_n 
			&\le \left(\frac{tn}{T_n}+1\right)\int_0^{T_n} \big(S_{(\theta_n,T_n)}(s)-S_{(\theta,T_n)}(s)-(\theta_n-\theta)^\top\nabla_\theta S_{(\theta,T_n)}(s)\big)^2ds \\
			&= \left(\frac{tn}{T_n}+1\right) \abs{\theta_n-\theta}^2 \int_0^{T_n}\!\left(\frac{S_{(\theta_n,T_n)}(s)-S_{(\theta,T_n)}(s)-(\theta_n-\theta)^\top\nabla_\theta S_{(\theta,T_n)}(s)}{\abs{\theta_n-\theta}} \right)^2\! ds \\
	&\le \left(\frac{tn}{T/2}+1\right) H^2n^{-1}\int_0^{2T} \!\left(\frac{S_{(\theta_n,T_n)}(s)-S_{(\theta,T_n)}(s)-(\theta_n-\theta)^\top\nabla_\theta S_{(\theta,T_n)}(s)}{\abs{\theta_n-\theta}} \right)^2\! ds,
\end{align*}
where the leading factor is obviously bounded and the integral tends to $0$ because of the $\L^2$-continuity condition (S5).

Next, using the H{\"o}lder condition (S6), we have for sufficiently large $n\in\N$
\begin{align*}
	B_n &\le \abs{\theta_n-\theta}^2 \int_0^{tn} \abs{\nabla_\theta S_{(\theta,T_n)}(s)-\nabla_\theta S_{(\theta,T)}(s)}^2ds \\
	&\le H^2n^{-1} \left(\int_0^{t_0} \abs{\nabla_\theta S_{(\theta,T_n)}(s)-\nabla_\theta S_{(\theta,T)}(s)}^2ds + C (tn)^\beta \abs{T_n-T}^\alpha \right) \\
	&\le H^2n^{-1} \int_0^{t_0} \abs{\dot S_{(\theta,T_n)}(s)-\dot S_{(\theta,T)}(s)}^2ds + C H^{2+\alpha} t^\beta n^{\beta-(1+3\alpha/2)}.
\end{align*}
The particular conditions on $\alpha$ and $\beta$ make the second summand vanish for $n\to\infty$, while the first summand converges to $0$ because of (S5).

In order to estimate $C_n$, we make explicit use of the $C^2$-property (S1), which is readily translated into the condition that the mapping
\[
	(0,\infty)\ni T \mapsto S_{(\theta,T)}(s)
\]
is twice continuously differentiable for any $s\in(0,\infty)$. Consequently, for every $s \in (0,\infty)$ Taylor expansion provides a $\rho=\rho(s,\theta,T,T_n,h_n)$ between $T$ and $T_n$ such that for sufficiently large $n \in \N$
\begin{align*}
	\abs{S_{(\theta,T_n)}(s)-S_{(\theta,T)}(s)-(T_n-T) \d_T S_{(\theta,T)}}	&=\frac12(T_n-T)^2\abs{\d_{TT}S_{(\theta,T)}(s)_{|_{T=\rho}}} \\
	&=\frac12 \left(h_n^{d+1}n^{-3/2}\right)^2\abs{\frac{s^2}{\rho^4} S_\theta''\left(\frac{s}{\rho}\right)+\frac{2s}{\rho^3} S_\theta'\left(\frac{s}{\rho}\right)} \\
	&\le \frac12 H^2 n^{-3}\left( s^2\frac{\norm{S_\theta''}_\infty}{(T-n^{-3/2}H)^4} +s\frac{2\norm{S_\theta'}_\infty}{(T-n^{-3/2}H)^3}\right) \\
	&\le \tilde c n^{-3}(s^2+s)
\end{align*}
for some positive constant $\tilde c$ (not depending on $s$ or $n$) and thus
\[
	C_n \le \tilde c^2 n^{-6}\int_0^{tn}(s^2+s)^2ds \xrightarrow{n \to \infty} 0.
\]
So far, we have shown that the sequence of random variables $(U_n(t))_{n\in\N}$ is bounded by a deterministic sequence which goes to zero as $n \to \infty$. Via the Burkholder-Davis-Gundy inequality this yields
\[
	\E^{(\theta,T)}\left[ \sup_{s\le t} \abs{R_n(s)}^2\right] \le 4 \E^{(\theta,T)}\left[\langle R_n\rangle_t\right] = 4 \E^{(\theta,T)}\left[ U_n(t)\right] \xrightarrow{n \to \infty} 0,
\]
so $R_n(t)$ vanishes in probability under $\P^{(\theta,T)}$ for $n \to \infty$. Finally, the same is true for the last remainder variable $V_n(t)$, as by Cauchy-Schwarz
\[
	\abs{V_n(t)}^2\le U_n(t) h_n^\top F_n(t)h_n \le U_n(t) H^2 \abs{F_n(t)}. 
\]
Taking $t=1$ completes the proof.
\end{proof}

\begin{remark}
If the shape parameter $\theta$ is assumed to be known, our Theorem includes \cite[Theorem 1.1]{HK3} as a special case (only (H1) - (H3) and (S1) are actually needed in this situation). If on the other hand, the periodicity is known and the only parameter of interest is $\theta$, then our Theorem leads to the same conclusion as \cite[Theorem 2.1]{HK1} (note that other than in our Theorem, here Score and Fisher Information are written at a time scale given by multiples of the known periodicity $T$). There, the $\L^2$-smoothness conditions on the signal are formulated under what is the measure $\nu^{(\theta,T)}$ in our notation, which under (H2) makes them slightly weaker than (S4) - (S6). However, if (H2) holds (which is more or less the only verifiable condition for $\nu^{(\theta,T)}$ to be finite, as supposed in \cite{HK1} anyway) the most obvious way to verify these is using that $\nu^{(\theta,T)}$ is thus bounded from above by a constant multiple of Lebesgue's measure, so the difference of the assumptions is just of a very theoretical nature. The key to bringing these results together in the above Theorem is the H{\"o}lder condition (S6), which is crucial for dealing with the term $B_n$ in step 3.) of the proof. This is the only instant where (in contrast to the terms $A_n$ and $C_n$) we have to impose more than just 'joint smoothness', but a more specific relation of the interplay between $T$ and $\theta$. It should also be noted that (H2) is essential for this step, as it removes any randomness from the terms we effectively deal with. Otherwise even if we would reformulate (S4) - (S6) in $\L^2(\nu^{(\theta,T)})$, we could not treat this term with Lemma \ref{H(s)} due to the occurrence of different periodicities in the integrand.
\end{remark}


\begin{example}
Consider the case
\[
	b(x)=-\beta x \quad \text{for some $\beta>0$,} \qquad \sigma(\cdot)\equiv\sigma>0 \quad \text{constant,}
\]
i.e. $\xi$ is a mean-reverting Ornstein-Uhlenbeck process with mean-reversion speed $\beta$ and time-dependent mean-reversion level $\beta^{-1} S_{(\theta,T)}(t)$. For the sake of simplicity, let us assume that $\sigma=1$. By \cite[Example 2.3]{HK2}, the periodic ergodicity assumption (H3) is fulfilled and we see that $\nu$ is simply Lebesgue's measure. In \cite{DFK} the authors think of $\beta$ as another unknown parameter, while they assume the periodicity $T$ to be fixed and known. In order to apply our results, we suppose that both $\beta$ and $T$ are fixed and known, while $\theta$ is to be estimated. The signal the authors consider is then the second one introduced in Example \ref{signal}. In this setting, we see that the Fisher Information $F$ is just the unit matrix and the Score is given by
\[
	\Delta_n=n^{-1/2} \left(\int_0^n \phi_i(s)dW_s\right)_{i=1,\ldots,d}. 
\]
Proposition 4.1 of \cite{DFK} implies that the rescaled estimation error $\sqrt{n}(\hat\theta_n-\theta)$ of the maximum likelihood estimator $\hat\theta_n$ is exactly the central statistic $Z_n=F^{-1}\Delta_n=\Delta_n$. Combining this with our Theorem \ref{LAN}, we see that in the sense of the Local Asymptotic Minimax Theorem (\cite[Theorem 7.12]{HoBo}) $\hat\theta_n$ is in fact optimal with rate $\sqrt{n}$ (cf. \cite[Theorem 2]{DFK}).
\end{example}

\begin{example}
More generally, for $\sigma\equiv 1$ and any measurable $b\colon\R\to\R$ the process $X=(X_t)_{t\ge0}$ defined by
\[
	X_t:=\xi_t - \int_0^t b(\xi_s)ds,
\]
is obviously a solution to the 'signal in white noise' equation
\begin{equation}\label{noise}
	dX_t=S_{(\theta,T)}(t)dt+dW_t, \quad t \in [0,\infty).
\end{equation}
We will now discuss some known results about this equation. Note that even if $\xi$ satisfies the ergodicity assumption (H3), $X$ does not. Ibragimov and Khasminskii treat the case where $\theta$ is fixed and known and $T$ is to be estimated (see \cite[p. 209-211]{Ibra}). They show asymptotic normality and efficiency for the maximum likelihood and Bayesian estimators with a normalization factor that coincides asymptotically with  
\[
	(\delta_n)_{d+1,d+1}^{-1} \left( F_{d+1,d+1} \right)^{-1/2},
\]
when translated into our notation (note that they use a different parametrization: 'our $T$' takes the place of 'their $\vartheta^{-1}$', explaining the different constants appearing). So both rate and limit variance are the right ones in the sense of the Local Asymptotic Minimax Theorem. Golubev (\cite{Golubev}, or see \cite{CLM} for a more detailed probabilistic explanation) gives an estimator for $T$ under unknown infinite-dimensional $\theta$ (the vector of the Fourier-coefficients of the signal) which he proves to be asymptotically normal and efficient, where the normalization factor is (when translated into our notation) given by
\[
	n^{3/2} \left( \frac{1}{12 T^4} \int_0^1 (S_\theta'(s))^2ds \right)^{1/2}= (\delta_n)_{d+1,d+1}^{-1} \left(\frac14 F_{d+1,d+1} \right)^{-1/2}.
\]
So while the rate is indeed $\delta_n$, the limit variance for Golubev's estimator apparently differs from the optimal value by a factor of 4. This is due to the fact that he studies a slightly different model in which the driving Brownian motion is two-sided and the process is observed over time intervals $[-n/2,n/2]$ and not $[0,n]$. This can be interpreted as two independent 'signal in white noise' models $X^{(1)},X^{(2)}$ each being observed over the interval $[0,n/2]$, where $X^{(1)}$ follows \eqref{noise} and $X^{(2)}$ follows \eqref{noise} with the signal replaced by the same signal run backwards in time. Obviously, $X^{(1)}$ and $X^{(2)}$ both generate the same Fisher Information $F_{d+1,d+1}(1/2)$, using the notation of the proof of Theorem \ref{LAN}. As a consequence of the independence structure, the Fisher Information in the experiment arising from observation of $(X^{(1)},X^{(2)})$ indeed turns out to be
\[
	2 \cdot F_{d+1,d+1}(1/2) = 2 \cdot (1/2)^3F_{d+1,d+1}(1)= \frac14 F_{d+1,d+1}.
\]
\end{example}

\end{document}